\documentclass[finalsize,finaltitle,draft]{zaa}

\newcommand{\eps}{\varepsilon}

\newcommand{\epsi}{\varepsilon}

\begin{document}

\title{On the behavior of periodic solutions of \\planar
autonomous Hamiltonian systems \\ with  multivalued periodic
perturbations}

\runtitle{Behavior of periodic solutions of perturbed Hamiltonian
systems}

\author{Oleg Makarenkov, Luisa Malaguti and Paolo Nistri $\,^*$}
\runauthor{Oleg Makarenkov et al.}

\def\thefootnote{\fnsymbol{footnote}}
\footnotetext[1]{Corresponding author.}

\address{Department of Mathematics, Imperial College London,
London, SW7 2AZ, UK; e-mail: o.makarenkov@imperial.ac.uk}

\address{Dipartimento di Scienze e Metodi dell'Ingegneria, Universit\`a di Modena e Reggio Emilia, 42100 Reggio Emilia,
Italy; e-mail: luisa.malaguti@unimore.it}

\address{Dipartimento di Ingegneria dell'Informazione,
Universit\`a di Siena, 53100 Siena, Italy; e-mail:
pnistri@dii.unisi.it}


\abstract{Aim of the paper is to provide a method to analyze the
behavior of $T$-periodic solutions $x_\eps, \eps>0$, of a perturbed
planar Hamiltonian system near a cycle $x_0$,  of smallest period
$T$, of the unperturbed system. The perturbation is represented by a
$T$-periodic multivalued map which vanishes as $\eps\to0$. In
several problems from nonsmooth mechanical systems this multivalued
perturbation comes from the Filippov regularization of a nonlinear
discontinuous $T$-periodic term. \noindent Through the paper,
assuming the existence of a $T$-periodic solution $x_\eps$ for
$\eps>0$ small, under the condition that $x_0$ is a nondegenerate
cycle of the linearized unperturbed Hamiltonian system we provide a
formula for the distance between any point $x_0(t)$ and the
trajectories $x_\eps([0,T])$ along a transversal direction to
$x_0(t).$}





\maketitle

\newcommand{\abs}[1]{\lvert#1\rvert}

\section{Introduction}

Let $x_0$ be a $T$-periodic cycle of the Hamiltonian system
\begin{equation}\label{np}
  \dot x= f(x),
\end{equation}
where $f\in C^1(\mathbb{R}^2,\mathbb{R}^2).$ In the recent
monographs \cite{awr} and \cite{awr1}, on the basis of numerical
simulations, was heuristically shown that the subharmonic Melnikov's
method  (\cite{mel}, \cite[Chapter 4, \S6]{guck}) correctly predicts
the existence of $T$-periodic solutions $x_\eps$ of the differential
inclusion
\begin{equation}\label{ps}
  \dot x\in f(x) +\eps g(t,x,\eps),
\end{equation}
where $g:\mathbb{R}\times\mathbb{R}^2\times[0,1]\to K(\mathbb{R}^2)$
is a multivalued map taking the values in the set $K(\mathbb{R}^2)$
of nonempty compact and convex sets of $\mathbb{R}^2$. Sufficient
conditions for the local and global existence of at least an
absolutely continuous solution of (\ref{ps}) starting from any
initial condition can be found in (\cite[Chapter 2]{AC}).


\noindent In \cite{awr} and \cite{awr1} the authors have
experimentally observed that if $\theta_0$ is a simple zero of the
subharmonic Melnikov's bifurcation function then (\ref{ps})
possesses a $T$-periodic solution $x_\eps$ such that
\begin{equation}\label{conv}
   x_\eps(t)\to x_0(t+\theta_0)\quad{\rm as}\ \eps\to 0, \quad{\rm
   uniformly\
   in}\ t\in[0,T].
\end{equation}

A theoretical justification of this result can be provided along the
lines of the papers \cite{fec}, \cite{nach} and \cite{jmaa}. In this
paper we do not provide conditions to ensure the existence of
$T$-periodic solutions $x_\eps$, for $\eps>0$ small, instead we want
to evaluate the distance between any point $x_0(t)$ and the curve
$x_\eps([0,T])$ providing in this way a tool to study the behavior
of the $T$-periodic solutions of (\ref{ps}) near $x_0.$ This tool,
together with the method based on the Melnikov's bifurcation
function mentioned above, permits to perform a complete analysis
both for the existence and the behavior near the cycle $x_0$ of the
$T$-periodic solutions $x_\eps$ to (\ref{ps}).

Since in this paper the existence of $T$-periodic solutions $x_\eps$
of (\ref{ps}) is assumed, we only require to the multivalued map $g$
the minimal regularity assumptions needed for our analysis. In fact,
through the paper we only assume that the map
$g:\mathbb{R}\times\mathbb{R}^2\times[0,1]\to K(\mathbb{R}^2)$ is
measurable or upper semicontinuous.

The interest of considering multivalued perturbation of system
(\ref{np}) is mainly related to the necessity, encountered in the
applications, to deal with  perturbations, having jump
discontinuities, of Hamiltonian autonomous systems. In fact, many
physical problems are modeled by ordinary differential equations
with discontinuous right hand side whose regularization produces a
multivalued map (see for instance \cite{fil1} and \cite{AC}). Among
them we like to cite the study of the self-sustained oscillations
induced by friction in one-degree of freedom mechanical systems.
This problem gives rise to a planar Hamiltonian system perturbed by
a periodic perturbation of small amplitude with jump
discontinuities, compare e.g. \cite[Chapter 15]{awr} where the
analysis was heuristically performed by means of the Melnikov
method.

The paper is organized as follows. In Section 2 assuming that the
linearized system
\begin{equation}\label{ls}
  \dot y=f'(x_0(t))y
\end{equation}
possesses a not $T$-periodic solution, in this case following
Rhouma-Chicone \cite{chic} $x_0$ is said to be nondegenerate, we
show the existence of a family $\{\Delta_\eps\}_{\eps>0}$ of real
numbers with $\Delta_\eps\to 0$ as $\eps\to 0$ such that
\begin{equation}\label{sqr1}
  \frac{\left\|x_\epsi(t+\Delta_\epsi)- {x}_0(t)\right\|}{\epsi}\le const\quad{\rm
  for\ any\ }t\in[0,T]{\rm\ and\ any\ }\epsi>0.
\end{equation}
This property has been already established by the authors in
\cite{jmaa} and \cite{cpaa} in the case when $x_0$ is an isolated
limit cycle and $g$ in (\ref{ps}) is a singlevalued continuous
function. In section~3 we employ property (\ref{sqr1}) together with
a suitably defined multivalued function $M^\bot\in
C^0(\mathbb{R},\mathbb{R})$ to obtain
\begin{equation}\label{fa}
  x_\eps(t+\Delta_\eps)-x_0(t)\in\eps M^\bot(t)y(t)+\alpha_\eps(t)\dot x_0(t)+o(\eps),
\end{equation}
where $y$ is a not $T$-periodic solution of the linearized system
(\ref{ls}) and $\alpha_\eps(t)$ is a scalar function  infinitesimal
as $\eps\to 0$ of order greater or equal to $1$. The function
$\alpha_\eps(t)$ is given in the formula (\ref{scalar}) of the
paper. The formula to represent the function $M^\bot$ is provided in
Section~3, thus (\ref{fa}) gives an explicit formula for the
distance between the trajectories $x_0$ and $x_\eps$ along a
transversal direction to $x_0$. Finally, in Section~4 we specialize
the formula for $M^\bot$ in the case when the Hamiltonian system
(\ref{np}) possesses symmetry properties, as often is the case in
the applications.

\section{Evaluation of the distance
between the periodic solutions of the perturbed system and the cycle
of the unperturbed one}

\noindent In this Section we establish the validity of inequality
(\ref{sqr1}) which is the starting point for (\ref{fa}). This result
does not depend on the perturbation term $g,$ indeed the only
property we need is the following one.

\begin{definition} (\cite{chic}) We say that the cycle ${x}_0$
of autonomous system (\ref{np}) is nondegenerate if the linearized
system (\ref{ls}) has a not $T$-periodic solution.
\end{definition}

If (\ref{np}) is Hamiltonian then the nondegeneracity of $x_0$
implies that the period $T$ of $x_0$ is noncritical (compare
\cite{crit1}).

\begin{definition} (\cite[Definition~2.2.1]{kunze}) A function $x:[0,T]\to\mathbb{R}^2$ is
said to be a solution of the differential inclusion (\ref{ps}) on
$[0,T]$ if $x$ is absolutely continuous and the inclusion in
(\ref{ps}) holds for almost all (a.a.) $t\in[0,T].$
\end{definition}

\begin{definition} (\cite[Definition~1.3.1]{koz}) For any $\eps>0$ the multivalued
map $g(\cdot,\eps):\mathbb{R}\times\mathbb{R}^2\to K(\mathbb{R}^2)$
is said to be measurable if, for any open $V\subset\mathbb{R}^2$,
the set $g^{-1}(V,\eps):=\{ (t,x)\in \mathbb{R}\times\mathbb{R}^2 \,
: \, g(t,x, \eps) \cap V \ne \emptyset \}$ is measurable.
\end{definition}

\vskip0.2truecm We assume the following condition.

\vskip0.3truecm\noindent{\bf (H)}$\;-\;$ for any bounded set
$B\subset \mathbb{R}^2$ there exists $\mu_B\in
L^\infty_{loc}(\mathbb{R})$ such that
$$
\Vert g(t,x,\eps) \Vert:= \mbox{sup}\{\Vert y \Vert: y\in
g(t,x,\eps)\} \le \mu_B(t)
$$
for all $t \in \mathbb{R}, \, x \in B$ and  $\eps \in [0,1].$

\vskip0.3truecm Note that the notion of nondegenerate cycles has
been used in \cite{jmaa} and \cite{cpaa} in a stronger sense, i.e.
$x_0$ is called nondegenerate if the linearized system (\ref{ls})
has only one characteristic multiplier equal to $+1$.

\noindent In order to introduce the family
$\{\Delta_\epsi\}_{\eps>0},$ following \cite{cpaa}, we define a
surface $S\in C(\mathbb{R},\mathbb{R}^2)$ as follows
\begin{equation}\label{suf}
\begin{array}{lll}
S(v)&=&\Omega(T,0,h(v)),\\
h(v)&=&{x}_0(0)+A_1 v,
\end{array}
\end{equation}
where $\Omega(\cdot,t_0,\xi)$ is the solution of (\ref{np})
satisfying $\Omega(t_0,t_0,\xi)=\xi$ and $A_1$ is an arbitrary
$2\times 1$ vector such that the $2 \times 2$ matrix $(\dot
x_0(0),A_1)$ is nonsingular.

\vskip0.3truecm The following result shows that the surface $S$
intersects $x_0$ transversally.

\begin{lemma} (\cite[Lemma~2.2]{cpaa}) \label{conv_lem2}
Assume $f\in C^1(\mathbb{R}^2,\mathbb{R}^2).$ Let $x_0$ be a
nondegenerate $T$-periodic cycle of (\ref{np}). Then
$\dot{{x}}_0(0)\not\in S'(0)(\mathbb{R}).$
\end{lemma}

\vskip0.3truecm Using the previous Lemma we can prove the following
result.

\begin{lemma}\label{conv_cor1}
Assume $f\in C^1(\mathbb{R}^2,\mathbb{R}^2)$ and that
$g:\mathbb{R}\times\mathbb{R}^2\times[0,1]\to K(\mathbb{R}^2)$ is
measurable and  satisfying (H). Let $x_0$ be a nondegenerate
$T$-periodic cycle of (\ref{np}). Let $x_\varepsilon$ be a
$T$-periodic solution to perturbed system (\ref{ps}) satisfying
$$\|x_\eps(t)-x_0(t)\|\to 0$$ as $\eps\to 0$ uniformly with respect to
$t\in\mathbb{R},$ then there exists $\varepsilon_0>0$ and $r_0>0$
such that for any $\varepsilon\in(0,\varepsilon_0]$ the equation
$x_\varepsilon(\Delta)=S(v)$ has a unique solution
$(\Delta_\varepsilon,v_\varepsilon)$ in
$[-r_0,r_0]\times\{v\in\mathbb{R}:|v|\le r_0\}.$ Moreover, the
functions $\varepsilon\to\Delta_\varepsilon,$  $\varepsilon\to
v_\varepsilon$ are continuous at $\varepsilon=0$ with $\Delta_0=0$
and $v_0=0.$
\end{lemma}

In the case when $g$ in (\ref{ps}) is singlevalued and continuous
Lemma~\ref{conv_cor1} is a simple consequence of
Lemma~\ref{conv_lem2} (\cite[Corollary~2.3]{cpaa}). In the present
case of $g$ multivalued map we should provide a proof.

\

\begin{proof} Define the function $F:
\mathbb{R}^2\times[0,1]\to \mathbb{R}^2$ as follows
$$
  F((t,v),\eps)=x_\eps(t)-S(v),
$$
then $F((0,0),0)=0.$  Our assumptions and definitions (\ref{suf})
guarantee that $F$ is a continuous function at the points
$\mathbb{R}^2\times\{0\}.$ Since $F(\cdot,0)$ is differentiable at
$(0,0)$ and $F'_{(t,v)}((0,0),0)=(\dot x_0(0),-S'(0))$ is
nonsingular by Lemma 1, then there exists $r_0>0$ such that
$$
  d(F(\cdot,0),[-r_0,r_0]\times[-r_0,r_0],0)\not=0,
$$
here $d(\Phi,V,0)$ denotes the topological degree of the map $\Phi$
in the set $V$ with respect to $0$.

\noindent Therefore, there exists $\eps_0>0$ such that
$$
  d(F(\cdot,\eps),[-r_0,r_0]\times[-r_0,r_0],0)\not=0,\quad{\rm for\
  any\ }\eps\in[0,\eps_0].
$$
This implies that for any $\eps\in[0,\eps_0],$ by the solution
property of the topological degree, there exists at least one pair
$(\Delta_\eps,v_\eps)\in[-r_0,r_0]\times[-r_0,r_0]$ such that
$x_\eps(\Delta_\eps)-S(v_\eps)=0.$

Let us show that this solution is unique in
$[-r_0,r_0]\times[-r_0,r_0]$ provided that $r_0>0$ and $\eps_0>0$
are sufficiently small. Assume the contrary, hence there exist
$\eps_k\to 0$ as $k\to\infty$ and
$(\widetilde{\Delta}_{\eps_k},\widetilde{v}_{\eps_k})\to (0,0)$ as
$k\to\infty$ such that
$$
  x_{\eps_k}(\widetilde{\Delta}_{\eps_k})-S(\widetilde{v}_{\eps_k})=0\quad{\rm
  and}\quad
  (\widetilde{\Delta}_{\eps_k},\widetilde{v}_{\eps_k})\not=(\Delta_{\eps_k},v_{\eps_k}),\quad{\rm
  for\ any\ }k\in\mathbb{N}.
$$
Since $S:[-r_0,r_0]\to S([-r_0,r_0])$ is invertible then
$(\widetilde{\Delta}_{\eps_k},\widetilde{v}_{\eps_k}) \not=
(\Delta_{\eps_k},v_{\eps_k})$ implies
$\widetilde{\Delta}_{\eps_k}\not=\Delta_{\eps_k}$,  say
$\widetilde{\Delta}_{\eps_k}<\Delta_{\eps_k}$. On the other hand
$\dot x(0)\not= 0$ and so we can assume $\widetilde{v}_{\eps_k}\not=
v_{\eps_k}$. For any $v_1,v_2\in\mathbb{R}^2$ we define
$\angle(v_1,v_2)$ as follows
$$
  \angle(v_1,v_2)=\arccos \frac{\left<v_1,v_2\right>}{\|v_1\|\cdot\|v_2\|}.
$$
Then we have
$$
  \angle(x_{\eps_k}(\Delta_{\eps_k})-x_{\eps_k}(\widetilde{\Delta}_{\eps_k}),\dot
  x_0(0))=\angle(S(v_{\eps_k})-S(\widetilde{v}_{\eps_k}),\dot
  x_0(0)).
$$
Passing to a subsequence if necessary we have that
$\left\{\dfrac{v_{\eps_k}-\widetilde{v}_{\eps_k}}{|v_{\eps_k}-\widetilde{v}_{\eps_k}|}\right\}_{k=1}^\infty$
converges. Denote by $q\in\mathbb{R}, |q|=1,$ the limit of this
sequence. Then
$$
  \angle(S(v_{\eps_k})-S(\widetilde{v}_{\eps_k}),\dot
  x_0(0))\to\angle(S'(0)q,\dot x_0(0))\quad{\rm as}\ k\to\infty,
$$
with $\angle(S'(0)q,\dot x_0(0))\not=0,$ since, by Lemma 1, $\dot
x(0)\not\in S'(0)(\mathbb{R})$.  Therefore, there exists $\alpha>0$
such that
\begin{equation}\label{A}
  \left|\angle(x_{\eps_k}(\Delta_{\eps_k})-x_{\eps_k}(\widetilde{\Delta}_{\eps_k}),\dot
  x_0(0))\right|\ge\alpha>0,\quad{\rm for\ any\ }k\in\mathbb{N}.
\end{equation}
Since $t\to x_{\eps_k}(t)$ is a solution of (\ref{ps}) then, by
Filippov's lemma (\cite{fil} or \cite[Theorem~1.5.10]{obu}), there
exists a singlevalued measurable function
$h_{\eps_k}:[0,T]\to\mathbb{R}^2$ such that
\begin{eqnarray*}
  \dot x_{\eps_k}(t) &=&f(x_{\eps_k}(t))+{\eps_k} h_{\eps_k}(t),\quad{\rm for \ a.a.\
  }t\in[0,T],\\
  h_{\eps_k}(t)&\in & g(t,x_{\eps_k}(t),{\eps_k}),\quad{\rm for\ a.a.\
  }t\in[0,T].
\end{eqnarray*}
Therefore
$$
  x_{\eps_k}(\Delta_{\eps_k})-x_{\eps_k}(\widetilde{\Delta}_{\eps_k})=
  \int_{\widetilde{\Delta}_{\eps_k}}^{\Delta_{\eps_k}}
  f(x_{\eps_k}(\tau))d\tau+\eps_k
  \int_{\widetilde{\Delta}_{\eps_k}}^{\Delta_{\eps_k}}h_{\eps_k}(\tau)d\tau.
$$
Due to the uniform convergence of $x_\eps$ to  $x_0$ as $\eps\to 0$
we have that
$$\mbox{sup}_{k\in N}\{\|x_{\eps_k}(\tau)\|: \tau\in
[0, T]\}<\infty$$
thus the assumptions on $f$ and $g$ permit to
conclude that
$$
  \angle(x_{\eps_k}(\Delta_{\eps_k})-x_{\eps_k}(\widetilde{\Delta}_{\eps_k}),\dot
  x_0(0))\to \angle(f(x_0(0)),\dot x_0(0)) \quad{\rm as\ }k\to
  \infty,
$$
hence $\angle(f(x_0(0)),\dot x_0(0))=0$ since $f(x_0(0))=\dot
x_0(0)$. This is a contradiction with (\ref{A}) and so the proof is
complete.\end{proof}

 \vskip0.3cm

We are now in the position to prove inequality (\ref{sqr1}).

\begin{theorem}\label{thmI} Assume $f\in
C^1(\mathbb{R}^2,\mathbb{R}^2)$ and
$g:\mathbb{R}\times\mathbb{R}^2\times[0,1]\to K(\mathbb{R}^2)$ is
measurable and  satisfying (H).
 Let $x_\varepsilon$ be a $T$-periodic solution to the
perturbed system (\ref{ps}) satisfying
\begin{equation}\label{conv}
  \|x_\varepsilon(t)- {x}_0(t)\|\to 0\mbox{\quad as}\quad\varepsilon\to
  0
\end{equation}
uniformly with respect to $t\in[0,T],$ where ${x}_0$ is a
nondegenerate $T$-periodic cycle of the unperturbed system
(\ref{np}). Let $\epsi_0>0$ and
$\{\Delta_\epsi\}_{\epsi\in(0,\epsi_0]}\subset \mathbb{R}$ be as in
Lemma~\ref{conv_cor1}. Then there exists $M>0$ such that
\begin{equation}\label{IN}
\|x_\varepsilon(t+\Delta_\varepsilon)- {x}_0(t)\|\le
M\epsi\quad\mbox{ for\ any\ }t\in[0,T] \mbox{ and\ any\ }\
\epsi\in(0,\epsi_0].
\end{equation}
\end{theorem}

\begin{proof} In the sequel $\epsi\in(0,\epsi_0]$ and
$\tau\in[0,T].$ Consider the change of variables
$\nu_\varepsilon(\tau)=\Omega(0,\tau,x_\varepsilon(\tau+\Delta_\varepsilon))$
in system (\ref{ps}). Observe that
\begin{equation}\label{PP}
x_\varepsilon(\tau+\Delta_\varepsilon)=\Omega(\tau,0,\nu_\varepsilon(\tau)).
\end{equation}
Taking the derivative in (\ref{PP}) with respect to $\tau$ we obtain
\begin{equation}\label{ob1}
 \dot x_\varepsilon(\tau+\Delta_\varepsilon)=
 f(\Omega(\tau,0,\nu_\varepsilon(\tau)))+\Omega'_{\xi}(\tau,0,
 \nu_\varepsilon(\tau)) \dot \nu_\varepsilon(\tau).
\end{equation}
On the other hand from (\ref{ps}) we have
\begin{equation}\label{ob2}
  \dot
  x_\varepsilon(\tau+\Delta_\varepsilon)\in
  f(\Omega(\tau,0,\nu_\varepsilon(\tau)))+\varepsilon
  g(\tau+\Delta_\varepsilon,
  \Omega(\tau,0,\nu_\varepsilon(\tau)),\varepsilon).
\end{equation}
Since $\Omega'_{\xi} (\tau,0,\nu_\varepsilon(\tau))$ is the
fundamental matrix of a linear system thus it is invertible, then
from (\ref{ob1}) and (\ref{ob2}) it follows
$$
  \dot \nu_\varepsilon(\tau)\in\varepsilon\left(\Omega'_{\xi}
  (\tau,0,\nu_\varepsilon(\tau))\right)^{-1}g(\tau+\Delta_\varepsilon
  ,\Omega(\tau,0,\nu_\varepsilon(\tau)),\varepsilon),
$$
and
$$
  \nu_\varepsilon(0)=x_\varepsilon(\Delta_\varepsilon)=x_\varepsilon(T+\Delta_\varepsilon)=
  \Omega(T,0,\nu_\varepsilon(T)).
$$
Since $g$ is measurable then again by Filippov's lemma there exists
a measurable singlevalued function $h_\eps:[0,T]\to\mathbb{R}^2$
such that
$$
  h_\eps(\tau)\in \left(\Omega'_{\xi}
  (\tau,0,\nu_\varepsilon(\tau))\right)^{-1}g(\tau+\Delta_\varepsilon
  ,\Omega(\tau,0,\nu_\varepsilon(\tau)),\varepsilon),\quad{\rm
  for\ a.a.\ }\tau\in[0,T]
$$
and
$$
  \dot \nu_\eps(\tau)=\eps h_\eps(\tau),\quad{\rm
  for\ a.a.\ }\tau\in[0,T].
$$
Therefore, $h_\eps\in L^\infty([0,T],\mathbb{R}^2)$ and

\begin{equation}\label{ob3}
  \nu_\varepsilon(\tau)=\Omega(T,0,\nu_\varepsilon(T))+\varepsilon\int\limits_0^\tau
  h_\eps(s)
  ds,\quad{\rm for\ any\ }\tau\in[0,T].
\end{equation}
Since, for any $\tau\ge 0,$ $\nu_\varepsilon(\tau)\to {x}_0(0)$ as
$\varepsilon\to 0$ we can write $\nu_\varepsilon(\tau)$ in the
following form
\begin{equation}\label{rep}
  \nu_\varepsilon(\tau)={x}_0(0)+\varepsilon
  \mu_\varepsilon(\tau).
\end{equation}
We now prove that the functions $\mu_\varepsilon$ are bounded on
$[0,T]$ uniformly with respect to $\varepsilon\in(0,\varepsilon_0].$
For this, we first subtract ${x}_0(0)$ from both sides of
(\ref{ob3}), with $\tau=T,$ obtaining
\begin{eqnarray}
\varepsilon\mu_\varepsilon(T)&=&\varepsilon\,\Omega'_{\xi}(T,0,{x}_0(0))
\mu_\varepsilon(T)+o(\varepsilon\mu_\varepsilon(T))
+\varepsilon\,\int_0^T h_\eps(s)
  ds,\label{ob4}
\end{eqnarray}
where, from (\ref{rep}),
$\dfrac{o(\epsi\mu_\epsi(T))}{\|\epsi\mu_\epsi(T)\|}\to 0$ as
$\epsi\to 0.$

\vskip0.2truecm\noindent Since $x_\varepsilon(\Delta_\varepsilon)\in
S\left(\{v\in\mathbb{R}:|v|\le r_0\right),$ then by
Lemma~\ref{conv_cor1} there exists $v_\varepsilon\in \mathbb{R},$
$|v_\epsi|\le r_0,$ such that
\begin{equation}\label{SP}
x_\varepsilon(\Delta_\varepsilon)= \Omega(T,0,x_0(0)+A_1 v_\epsi)
\end{equation}
and
\begin{equation}\label{convr}
v_\varepsilon\to 0\quad{\rm as}\quad \varepsilon\to 0.
\end{equation}
Now by using (\ref{SP}) we can represent
$\varepsilon\mu_\varepsilon(T)$ as follows
\begin{equation}\label{key}
\begin{array}{lll}
\varepsilon\mu_\varepsilon(T) & = &
\nu_\varepsilon(T)-{x}_0(0)=\Omega(0,T,x_\varepsilon
(\Delta_\varepsilon))-{x}_0(0)=\\
&= & \Omega(0,T,\Omega(T,0,x_0(0)+A_1 v_\epsi))-{x}_0(0)=A_1
v_\varepsilon.
\end{array}
\end{equation}
Therefore  (\ref{ob4}) can be rewritten as follows
\begin{eqnarray}
A_1 v_\varepsilon&=&\Omega'_{\xi}(T,0,{x}_0(0))A_1 v_\varepsilon+
o(A_1 v_\varepsilon)+\, \varepsilon\int_0^T h_\eps(s)
  ds.\label{bisp}
\end{eqnarray}
Let us show that there exists $M_1>0$ such that
\begin{equation}\label{estt}
  \vert v_\varepsilon\vert \le \varepsilon
  M_1,\qquad\mbox{for any}\quad \varepsilon\in(0,\varepsilon_0].
\end{equation}
Arguing by contradiction we assume that there exist sequences
$\{\varepsilon_k\}_{k\in\mathbb{N}}\subset(0,\epsi_0],$
$\varepsilon_k\to 0$ as $k\to\infty,$ such that
$|v_{\varepsilon_k}|=\varepsilon_k c_k,$ where $c_k\to\infty$ as
$k\to\infty.$ Let
$q_k=\dfrac{v_{\varepsilon_k}}{|v_{\varepsilon_k}|},$ then from
(\ref{bisp}) we have
\begin{eqnarray}\label{ff}
 & & A_1 q_k=\Omega'_{\xi}(T,0,{x}_0(0))A_1 q_k+\frac{o(A_{1}v_{\varepsilon_k})}
 {|v_{\varepsilon_k}|}+\frac{1}{c_k}\,\int_0^{T}
 h_{\eps_k}(s)
  ds,
\end{eqnarray}
where $\dfrac{o(A_{1}v_{\varepsilon_k})}{|v_{\varepsilon_k}|}\to 0$
as $k\to \infty,$ in fact
$\dfrac{o(A_{1}v_{\varepsilon_k})}{|v_{\varepsilon_k}|}=
\dfrac{o(A_{1}v_{\varepsilon_k})}{\|A_{1}v_{\varepsilon_k}\|}\cdot
\dfrac{\|A_{1}v_{\varepsilon_k}\|}{|v_{\varepsilon_k|}}.$

\vskip0.2truecm\noindent Let $ B=\{ v_{\eps}(\tau) \, : \, \tau\in
[0,T], \, \eps \in [0,1]\}. $ The continuity of $\Omega$ and
condition \eqref{conv} imply that $B$ is bounded. Since also $
(\Omega_{\xi}')^{-1}$ is continuous, we can find $\Lambda >0$
satisfying $\left \Vert \left( \Omega_{\xi}'(T, 0, v_{\eps}(\tau)
\right)^{-1} \right \Vert \le \Lambda$ for any $\tau \in [0,T]$ and
any $\eps \in [0, \eps_{0}]$. Therefore, from assumption (H) we
obtain
\begin{equation}\label{eq:h}
\left \Vert \int_0^T h_{\eps}(s) \, ds \right\Vert \le \eps \Lambda
\int_0^T \mu_{B}(s+\Delta_{\eps})\, ds < +\infty, \quad \mbox{for }
\eps\in [0, \eps_0].
\end{equation}
Without loss of generality we may assume that the sequence
$\{q_k\}_{k\in\mathbb{N}}\subset \mathbb{R}$ converges, let
$q_0=\lim_{k\to\infty}q_k$ with $|q_0|=1$. By passing to the limit
as $k\to\infty$ in (\ref{ff}) we have that
$$
A_{1}q_0=\Omega'_{\xi}(T,0,{x}_0(0))A_{1} q_0.
$$
Therefore $A_{1}q_0$ is the initial condition of a $T$-periodic
solution to (\ref{ls}). On the other hand the cycle $x_0$ is
nondegenerate, hence $A_{1}q_0$ is linearly dependent with $\dot
x_0(0)$ contradicting the choice of $A_{1}.$ Thus (\ref{estt}) is
true for some $M_1>0.$ From (\ref{rep}) and the fact that
$\nu_\varepsilon(0)=x_\varepsilon(\Delta_\varepsilon)$ we have
\begin{equation}\label{DEC}
\begin{array}{lll}
 \|x_\epsi(\Delta_\epsi)-{x}_0(0)\|&=&\varepsilon\|\mu_\varepsilon(0)\|\le
 \epsi\|\mu_\epsi(T)\|+\|\epsi\mu_\epsi(T)-\epsi\mu_\epsi(0)\|=\\
 &=&\epsi\|\mu_\epsi(T)\|+\|\nu_\epsi(T)-\nu_\epsi(0)\|.
\end{array}
\end{equation}
From (\ref{ob3})  and \eqref{eq:h} we have that there exists $M_2>0$
such that
\begin{equation}\label{estt1}
  \|\nu_\epsi(T)-\nu_\epsi(0)\|\le \epsi M_2,\qquad{\rm for\ any\
  }\epsi\in(0,\epsi_0].
\end{equation}
Therefore combining (\ref{key}) with (\ref{estt})  and taking into
account (\ref{estt1}) we have from (\ref{DEC}) that
$$
  \|x_\epsi(\Delta_\epsi)-{x}_0(0)\|\le \epsi\|A_{1}\|
  M_1  +\epsi M_2,\qquad{\rm for\ any\
  }\epsi\in(0,\epsi_0].
$$
Since
$$
  \dot x_\epsi(t+\Delta_\epsi)\in f(x_\eps(t+\Delta_\eps))+\eps g(t+\Delta_\epsi,x_\epsi(t+\Delta_\epsi),\epsi)
$$
and $g$ is measurable then Filippov's lemma ensures the existence of
a measurable singlevalued function $m_\eps:[0,T]\to\mathbb{R}^2$
such that
$$
  \dot x_\epsi(t+\Delta_\epsi)= f(x_\eps(t+\Delta_\eps))+\eps
  m_\eps(t),\quad{\rm for\ a.a.\ }t\in[0,T]
$$
and
$$
  m_\eps(t)\in g(t+\Delta_\epsi,x_\epsi(t+\Delta_\epsi),\epsi),\quad{\rm for\ a.a.\
  }t\in[0,T].
$$
This allows to conclude that
$$
x_\epsi(t+\Delta_\epsi)-x_0(t)=x_\epsi(\Delta_\epsi)-x_0(0)+\int\limits_0^t(f(x_\epsi(s+\Delta_\epsi))-
f(x_0(s)))ds+\epsi\int\limits_0^t m_\eps(s) ds.
$$
Therefore,  there  exists a constant $M_3\ge 0$ such that, for any
$\epsi\in(0,\epsi_0],$ we have
\begin{equation}\label{GR}
\begin{array}{lll}
\|x_\epsi(t+\Delta_\epsi)-x_0(t)\|&\le& \left(\epsi\|A_{1}\|M_1+
 \epsi M_2\right)+\\
 & &+M_3\int\limits_0^t\|x_\epsi(s+\Delta_\epsi)-
x_0(s)\|ds+\epsi M_3,
\end{array}
\end{equation}
By means of the Gronwall-Bellman lemma, (compare e.g.
\cite[Chapter~II,  \S~11]{dem}), inequality (\ref{GR}) implies
\begin{equation*}
\|x_\epsi(t+\Delta_\epsi)-x_0(t)\|\le  \epsi\left(\|A_{1}\|
  M_1+ M_2+M_3\right){\rm e}^{M_3 T}\quad{\rm for\ any\
  }\epsi\in(0,\epsi_0].
\end{equation*}
and thus the proof is complete.
\end{proof}

\begin{remark} Observe that Theorem~\ref{thmI} does not require that
(\ref{np}) is a Hamiltonian system, indeed the crucial assumption is
that the linearized system (\ref{ls}) has a not $T$-periodic
solution.
\end{remark}

\section{First approximation formula for periodic solutions of the perturbed system}

Denote by $\widetilde{z}$ a non-trivial $T$-periodic solution of the
adjoint system
\begin{equation}\label{ss}
   \dot z=-(f'(x_0(t)))^*z.
\end{equation}
Let $t_*\in[0,T]$ such that
$$
  \widetilde{z}_1(t_*)=0,
$$
hence $\widetilde{z}_2(t_*)\not=0$. We begin the Section by studying
the behavior, as $\eps\to 0,$ of the scalar product
\begin{equation}\label{prod}
  \left<\widetilde{z}(t), \frac{x_\eps(t+\Delta_\eps)-x_0(t)}{\eps}\right>
\end{equation}
which is the starting point for deriving the first approximation
formula (\ref{fa}). To this end we denote by
$\widehat{z}=(\widehat{z}_1,\widehat{z}_2)$ any  solution of
(\ref{ss}) defined in $[0,T]$ linearly independent with
$\widetilde{z}$ and introduce the multivalued map $M^\bot:[0,T]\to
K(\mathbb{R})$ as follows
\begin{equation}\label{M}
\begin{array}{ll}
M^\bot(t)=& \{\gamma(t_*) \int_{t-T}^t
\left<-\widehat{z}(\tau),h(\tau)\right>d\tau: \\
 & h\in L^\infty([-T,T],\mathbb{R}^2),
  h(t)\in g(t,x_0(t),0)\ {\rm for\ a.a\
  }t\in[-T,T] \},
\end{array}
\end{equation}
where $\gamma(t_*) =
\dfrac{\widetilde{z}_2(t_*)}{\widehat{z}_2(T+t_*)-\widehat{z}_2(t_*)}.$

\vskip0.4truecm We can prove the following result.

\begin{theorem}\label{conv_thm1} Assume $f\in C^1(\mathbb{R}^2,\mathbb{R}^2)$
and $g:\mathbb{R}\times\mathbb{R}^2\times[0,1]\to K(\mathbb{R}^2)$
upper semicontinuous and satisfying (H). Let $x_\varepsilon$ be a
$T$-periodic solution to the perturbed system (\ref{ps}) such that
\begin{equation}\label{INp}
\|x_\varepsilon(t+\Delta_\varepsilon)- {x}_0(t)\|\le
M\epsi\quad\mbox{for\ any\ }t\in[0,T]\ \mbox{and any}\
\epsi\in(0,\epsi_0],
\end{equation}
where $\Delta_\epsi\to 0$ as $\epsi\to 0,$ $M$ and $\epsi_0$ are
positive constants  and $x_0$ is a nondegenerate cycle of the
Hamiltonian system (\ref{np}). Then
\begin{equation}\label{mainprop}
  \lim_{\epsi\to
0}\rho\left(\dfrac{1}{\epsi}\left<\widetilde{z}(t),x_\varepsilon(t+\Delta_\varepsilon)-
{x}_0(t)\right>,
\frac{\widetilde{z}_2(t_*)}{\widehat{z}_2(T+t_*)-\hat
z_2(t_*)}M^\bot(t)\right)=0
\end{equation}
uniformly with respect to $t\in[0,T],$ where for any
$v\in\mathbb{R}^n$ and $S\subset\mathbb{R}^n$ the distance
$\rho(v,S)$ is defined as $\rho(v,S)=\inf_{s\in S}\|v-s\|.$
\end{theorem}

To prove Theorem~\ref{conv_thm1} we need the following lemma.

\begin{lemma} \label{lemma1}
Assume that the $T$-periodic system
\begin{equation}\label{t1}
  \dot u =A(t)u,\quad u\in\mathbb{R}^2
\end{equation}
has the characteristic multiplier $+1$ of algebraic multiplicity 2.
Let us denote by $\widetilde {u}=(\widetilde {u}_1,\widetilde
{u}_2)$ a $T$-periodic solution of (\ref{t1}) such that
$$ \widetilde {u}_1(0)=0,\ \widetilde {u}_2(0)\not=0.$$
Denote by  $\widehat {u}=(\widehat {u}_1,\widehat {u}_2)$ any
solution of (\ref{t1}) satisfying
$$\widehat {u}_1(0)\not=0.$$ Then
$$
\widehat {u}(t+T)=\widehat {u}(t)+\frac{\widehat
{u}_2(T)-\widehat{u}_2(0)}{\widetilde {u}_2(0)}\widetilde
{u}(t),\quad{for\ any\ } t\in\mathbb{R}.
$$
\end{lemma}

\

This result has been proved in \cite[Lemma~4.2]{tmms} under the
additional assumption $\widehat {u}_2(0)=0$. Though it is immediate
to see that avoiding this assumption does not affect the proof of
\cite[Lemma~4.2]{tmms} at all we provide here a proof of Lemma~
\ref{lemma1} for a sake completeness.

\

\begin{proof} Denote by $X$ the fundamental matrix of system
(\ref{t1}) such that $X(0)=I$. Since
$$X(T)\left(\begin{array}{c}0\\1\end{array}\right)=
\left(\begin{array}{c}0\\1\end{array}\right),$$ then
$X(T)=\left(\begin{array}{cc}a&0\\b&1\end{array}\right)$ with
$a,b\in\mathbb{R}$. By our assumption $X(T)$ has two eigenvalues
equal to $+1,$ therefore
$$
X(T)=\left(\begin{array}{cc}1&0\\b&1\end{array}\right),
b\in\mathbb{R}.
$$
We have
\begin{eqnarray}
  X(t+T)\widehat {u}(0)&=&X(t)X(T)\widehat{
  u}(0)=X(t)\left(\begin{array}{cc}1&0\\b&1\end{array}\right)\widehat{
  u}(0)=\nonumber\\&=& X(t)\widehat {u}(0)+X(t)\left(\begin{array}{c}0\\b\widehat{
  u}_1(0)\end{array}\right)=
  X(t)\widehat {u}(0)+\dfrac{b\widehat {u}_1(0)}{\widetilde {u}_2(0)}
  \widetilde{
  u}(t).\nonumber
\end{eqnarray}
On the other hand
$$
  X(T)\widehat{
  u}(0)=\left(\begin{array}{cc}1&0\\b&1\end{array}\right)\widehat{
  u}(0)=\widehat {u}(0)+\left(\begin{array}{c}0\\b\widehat{
  u}_1(0)\end{array}\right),
$$
which implies $b\widehat {u}_1(0)=\widehat {u}_2(T)-\widehat
{u}_2(0).$ This completes the proof.
\end{proof}
We now prove Theorem~\ref{conv_thm1}.
\begin{proof} In what follows
$\epsi\in(0,\epsi_0],$  $t,\tau\in[-T,T]$ and $\widetilde z,$
$\widehat z$ are the functions introduced at the beginning of this
section. Let $A$ be a nonsingular $2\times 2$ matrix such that
\begin{equation}\label{Z0}
  \widehat{z}(0)^*\,A=(0,1).
\end{equation}
Let $Y(t)$ be the fundamental matrix of the linearized system
(\ref{ls}) with initial condition $Y(0)=A.$ Let
\begin{equation}\label{FR}
   Z(t)=\left(Y(t)^*\right)^{-1}
\end{equation}
and define $a_\epsi\in C([-T,T],\mathbb{R}^2)$ as follows
$$
  a_{\epsi}(t)=Z(t)^*\,\frac{x_\epsi(t+\Delta_\epsi)-{x}_0(t)}{\epsi}.
$$
Then we have
\begin{equation}\label{fw}
x_\epsi(t+\Delta_\epsi)-{x}_0(t)=\epsi Y(t)a_\epsi(t).
\end{equation}
In what follows by $o(\varepsilon),$ $\varepsilon>0,$ we will denote
a function, which may depend also on other variables, having the
property that $\dfrac{o(\varepsilon)}{\varepsilon}\to 0$ as
$\varepsilon\to 0$ uniformly with respect to these variables when
they belong to any bounded set. Since
$$
  \dot x_\eps(t+\Delta_\eps)\in f(x_\eps(t+\Delta_\eps))+\eps
  g(t+\Delta_\eps,x_\eps(t+\Delta_\eps),\eps),\quad{\rm for\ a.a.\
  } t\in\mathbb{R}
$$
then, again by Filippov's lemma there exists a measurable
singlevalued function $h_\eps:\mathbb{R}\to \mathbb{R}^2$ such that
\begin{equation}\label{subtr}
  \dot x_\eps(t+\Delta_\eps)= f(x_\eps(t+\Delta_\eps))+\eps
  h_\eps(t),\quad{\rm for\ a.a.\
  } t\in\mathbb{R}
\end{equation}
and
$$
  h_\eps(t)\in g(t+\Delta_\eps,x_\eps(t+\Delta_\eps),\eps),\quad{\rm for\ a.a.\
  } t\in\mathbb{R}.
$$
By subtracting (\ref{np}) where $x(t)$ is replaced by ${x}_0(t)$
from (\ref{subtr}) we obtain
\begin{eqnarray}\label{bis1}
  \dot x_\varepsilon(t+\Delta_\varepsilon)-\dot{
  {x}}_0(t)=f'({x}_0(t))(x_\varepsilon(t+\Delta_\varepsilon)-{x}_0(t))+\varepsilon
  h_\eps(t)+o_t(\varepsilon),
\end{eqnarray}
for a.a. $t\in[-T,T],$ here $\varepsilon\to o_t(\varepsilon)$ is
such that $ o_{t+T}(\cdot)=o_t(\cdot)$ for any $t\in\mathbb{R}.$ By
substituting (\ref{fw}) into (\ref{bis1}) we have
$$
  \epsi \dot Y(t)a_\epsi(t)+\epsi Y(t)\dot a_\epsi(t)=\epsi
 f'({x}_0(t))Y(t)a_\epsi(t)+\epsi h_\eps(t)+
  o_t(\varepsilon),
$$
for a.a. $t\in[-T,T].$ Since $f'({x}_0(t))Y(t)=\dot Y(t)$ the last
formula can be rewritten as follows
\begin{equation}\label{BS}
  \epsi Y(t)\dot a_\epsi(t)
 =\epsi h_\eps(t)+
  o_t(\varepsilon),\quad{\rm for\ a.a.\ }t\in[-T,T].
\end{equation}
By means of Perron's lemma \cite{per}, (see also Demidovich
\cite[Sec. III, \S 12]{dem}), formula (\ref{Z0}) implies that
\begin{equation}\label{IM}
  \widehat{z}(t)^*\, Y(t)=(0,1)\quad{\rm for\ any\ }t\in\mathbb{R}.
\end{equation}
Therefore, applying $\widehat z(t)^*$ to both sides of (\ref{BS}) we
have
$$
  \epsi(\dot a_{\epsi,2})(t)=\epsi \widehat{z}(t)^*\, h_\eps(t)
  +\widehat{z}(t)^*\, o_t(\varepsilon),\quad{\rm for\ a.a.\ }t\in[-T,T],
$$
where $a_{\epsi,2}(t)$ is the second component of the vector
$a_\epsi(t),$ and so
\begin{equation}\label{prob}
\begin{array}{rcl}
  a_{\varepsilon,2}(t)&=&a_{\varepsilon,2}(t_0)+\int\limits_{t_0}^t
  \left<\widehat{z}(\tau),h_\eps(\tau)\right>d\tau +\int\limits_{t_0}^t
  \left<\widehat{z}(\tau),\dfrac{o_\tau(\epsi)}{\epsi}\right>d\tau,
\end{array}
\end{equation}
for all $t, t_0 \in [-T,T]$. From (\ref{FR}) we have that
$Z(0)^*\,Y(0)=I.$ Therefore
$$
  \left(\left[Z(0)\right]_2\right)^* A=(0,1),
$$
where $[Z(0)]_2$ denotes the second column of $Z(0).$ Thus
$[Z(0)]_2=\widehat{z}(0).$ Therefore,
$$
  a_{\eps,2}(t)=\left<\widehat{z}(t),\frac{x_\eps(t+\Delta_\eps)-x_0(t)}{\eps}\right>.
$$
Since $\widehat{z}$ is linearly independent with $\widetilde{z}$
then
$$
  \widehat{z}_1(t_*)\not=0.
$$
Since system (\ref{np}) is Hamiltonian then the algebraic
multiplicity of the characteristic multiplier +1 of linearized
system (\ref{ls}) is equal to 2. By lemma~\ref{lemma1} we have
$$
  \widehat{z}(t)=\widehat{z}(t-T)+\frac{\widehat{z}_2(T+t_*)-\widehat{z}_2(t_*)}
  {\widetilde{z}_2(t_*)}\widetilde{z}(t)= \widehat{z}(t-T) + \dfrac{1}{\gamma(t_*)}\widetilde{z}(t),
$$
that implies
$$
  a_{\eps,2}(t_0)=a_{\eps,2}(t_0-T)+\frac{1}{\gamma(t_*)}
  \left<\widetilde{z}(t_0),\frac{x_\eps(t_0+\Delta_\eps)-x_0(t_0)}{\eps}\right>.
$$
Substituting the last formula into (\ref{prob}) we obtain
\begin{equation}\label{tildez}
\begin{array}{ll}
\int_{t_0}^{t_0-T}\left<\widehat{z}(\tau),h_\eps(\tau)\right>d\tau&=
-\dfrac{1}{\gamma(t_*)}\left<\widetilde{z}(t_0),
\dfrac{x_\eps(t_0+\Delta_\eps)-x_0(t_0)}{\eps}\right>-\\&\quad
-\int_{t_0}^{t_0 -T} < \widehat {z}(\tau), \,
\dfrac{o_{\tau}(\eps)}{\eps} > \, d\tau.
\end{array}
\end{equation}

We claim that
\begin{equation}\label{claim}
 \lim_{\eps\to
0}\rho\left(\gamma(t_*)
\int_{t_0}^{t_0-T}\left<\widehat{z}(\tau),h_\eps(\tau)\right>d\tau,
M^\bot(t)\right)=0,
\end{equation}
uniformly with respect to $t_0 \in [0,T]$, with $h_{\eps}$ defined
as in \eqref{subtr}. To prove this we observe that the subset of
$\mathbb{R}$ given by
\begin{eqnarray*}
M :&=& \{\gamma(t_*) \int_{-T}^{T}
\left<\widehat{z}(\tau),h(\tau)\right>d\tau \, : \\&& h \in
L^\infty([-T, T], \mathbb{R}^2) \mbox{ and } h(t) \in g(t, x_0(t),
0) \mbox{ for a.a. } t\in [-T, T] \}
\end{eqnarray*}
is nonempty and compact; hence, for each $\eps \in (0, \eps_0]$,
there exists $k_{\eps} \, : \mathbb{R} \to \mathbb{R}$ such that
$k_{\eps}(t) \in g(t, x_0(t), 0)$ for a.a. $t$ and
$$
\rho\left(\gamma(t_*)
\int_{-T}^{T}\left<\widehat{z}(\tau),h_\eps(\tau)\right>d\tau,
M\right)=|\gamma(t_*)|\left \vert
\int_{-T}^{T}\left<\widehat{z}(\tau),h_\eps(\tau)-k_\eps(\tau)\right>d\tau
\right \vert.
$$
The upper semicontinuity of $g$ in the bounded set
$$[-T, T] \times \{ x_{\eps}(t)\, : \, t \in [0,T], \, \eps \in [0, \eps_0]\}\times [0,
\eps_0]
$$
implies  that, given $\delta>0$, there exists $ \eps_1 \in (0,
\eps_0]$ such that, for all $\eps\in[0,\eps_1],$ we have that
\begin{equation}\label{si}
  h_\eps(t)\in g(t+\Delta_\eps,x_\eps(t+\Delta_\eps),\eps)\subset
  B_\delta(g(t,x_0(t),0)),\quad{\rm for\ a.a.\ }t\in[-T,T],
\end{equation}
Fix an arbitrary $\delta>0$ and let $ \eps_1 \in (0, \eps_0]$
satisfying (\ref{si}). Let $\eps\in[0,\eps_1]$ and $t_0 \in [0,T]$.
We obtain
$$
\begin{array}{ll}
&\rho\left(\gamma(t_*)
\int_{t_0-T}^{t_0}\left<\widehat{z}(\tau),h_{\eps}(\tau)\right>d\tau,
M^\bot(t)\right)  \le\\&\quad\le  |\gamma(t_*)|
\left\vert\int_{t_0-T}^{t_0}\left<\widehat{z}(\tau),h_{\eps}(\tau)-k_{\eps}(\tau)\right>d\tau
\right \vert \le \\
 &\quad\le|\gamma(t_*)|\int_{-T}^{T} \left \vert
<\widehat{z}(\tau),h_{\eps}(\tau)-k_{\eps}(\tau) > \right \vert
d\tau \le 2T|\gamma(t_*)|\delta \Vert \widehat{z} \Vert_C.
\end{array}
$$
which implies our assertion (\ref{claim}). According to
\eqref{tildez}, the proof is complete.
\end{proof}

\begin{remark} The assumption that (\ref{np}) is Hamiltonian ensures that the
linearized system (\ref{ls}) has a characteristic multiplier +1 of
algebraic multiplicity 2 and so the assumption of Lemma
\ref{lemma1}. Alternatively, we could directly assume that the
algebraic multiplicity of the characteristic multiplier +1 of
(\ref{ls}) is equal to 2. The latter is a bit more general. The same
consideration applies to Theorems \ref{thm3} and \ref{thm4} below.
\end{remark}

We have the following result.

\begin{lemma}\label{new1}
   Let $x_0$ be a nondegenerate $T$-periodic cycle of the Hamiltonian system
   (\ref{np}). Let $\widetilde{z}$ be any $T$-periodic solution of
   the adjoint system (\ref{ss}). Then
   \begin{equation}\label{co}\left<\dot
x_0(t),\widetilde{z}(t)\right>=0,\quad{\rm for\ any\ }
t\in\mathbb{R}.\end{equation}
\end{lemma}

\begin{proof} Let $t_*\in[0,T]$ be such that
$\widetilde{z}_1(t_*)=0.$ Let $\widehat{z}$ be any solution of
(\ref{ss}) linearly independent with $\widetilde{z}.$ Then from
Lemma~\ref{lemma1} we have
$$
  \left<\dot x_0(t),\widehat{z}(t+T)\right>=\left<\dot
  x_0(t),\widehat{z}(t)\right>+\frac{\widehat{z}_2(T+t_*)}{\widetilde{z}_2(t_*)}\left<\dot
  x_0(t),\widetilde{z}(t)\right>,\quad{\rm for\ any\
  }t\in\mathbb{R}.
$$
Perron's lemma \cite{per} implies that $\left<\dot
x_0(t),\widehat{z}(t+T)\right>=\left<\dot
  x_0(t),\widehat{z}(t)\right>$ for any $t\in\mathbb{R}$ and thus (\ref{co}). \end{proof}

Lemma~\ref{new1} allows the reader to better understand the
substantial difference between the situation when the cycle $x_0$ is
isolated, which is studied in \cite{jmaa} and \cite{cpaa} and the
present situation when the cycle is non-isolated. In fact, in
\cite{jmaa} and \cite{cpaa} it is shown that $\left<\dot
x_0(t),\widetilde{z}(t)\right>\not=0,$ for any $t\in\mathbb{R},$
which is the contrary of (\ref{co}).

\begin{remark}\label{remark3} Let $\widetilde z$ be any T-periodic solution of
the adjoint system (\ref{ss}) and $\widehat z$ any solution of
(\ref{ss}) linearly independent with $\widetilde z.$
Lemma~\ref{new1} ensures that $\left<\dot
x_0(t),\widetilde{z}(t)\right>=0$ for any $t\in \mathbb{R}$,
moreover from the Perron's Lemma $\left<\dot
x_0(t),\widehat{z}(t)\right>= \left<\dot
x_0(0),\widehat{z}(0)\right>\not=0$ for any $t\in \mathbb{R}$.
Without loss of generality we can assume that $\left<\dot
x_0(0),\widehat{z}(0)\right>=1.$

\noindent Let $y$ be the function defined by

$$
y(t)^* = \left(
\dfrac{-\widehat{z}_2(t)}{\mbox{det}(\widehat{z}(t),\widetilde{z}(t))},
\dfrac{\widehat{z}_1(t)}{\mbox{det}(\widehat{z}(t),\widetilde{z}(t))}\right)
$$
then
\begin{equation}\label{form}
 (\dot x_0(t),y(t))=\left(\begin{array}{c}\widehat{z}(t)^*\\
\widetilde{z}(t)^*\end{array}\right)^{-1}.
\end{equation}
is a matrix solution of the linearized system (\ref{ls})
(\cite[Chapter~III, \S12]{dem}).
\end{remark}

We can now formulate the following result.

\begin{theorem}\label{thm3} Assume $f\in
C^1(\mathbb{R}^2,\mathbb{R}^2)$ and
$g:\mathbb{R}\times\mathbb{R}^2\times[0,1]\to K(\mathbb{R}^2)$ upper
semicontinuous and satisfying (H). Let $x_\varepsilon$ be a
$T$-periodic solution to perturbed system (\ref{ps}) such that
\begin{equation*}
  \|x_\varepsilon(t)- {x}_0(t)\|\to 0\mbox{\quad as}\quad\varepsilon\to
  0
\end{equation*}
uniformly with respect to $t\in[0,T],$ where ${x}_0$ is a
nondegenerate $T$-periodic cycle of the Hamiltonian system
(\ref{np}). Let $\widetilde z,$  $\widehat z$ be as in Remark~
\ref{remark3} and $\dot x_0,$  $y$ as in (\ref{form}). Then there
exists a family $\{\Delta_\eps\}_{\eps>0}$ such that $\Delta_\eps\to
0$ as $\eps\to 0$ and
\begin{equation}\label{scalar}
  \lim_{\eps\to 0}\rho\left(x_\eps(t+\Delta_\eps)-x_0(t),
  \eps M^\bot(t)y(t)+ \left<\widehat{z}(t),x_\eps(t-\Delta_\eps)-x_0(t)\right>\dot
  x_0(t)\right)=0,
\end{equation}
uniformly with respect to $t\in[0,T].$

\end{theorem}

\vskip0.3cm


\

\begin{proof}

The proof of theorem~\ref{thm3} follows from the following
representation which is a consequence of (\ref{form})
\begin{eqnarray*}
  x_\eps(t+\Delta_\eps)-x_0(t)=&
  \left<\widetilde{z}(t),x_\eps(t+\Delta_\eps)-x_0(t)\right>y(t)+\\\quad&+
  \left<\widehat{z}(t),x_\eps(t+\Delta_\eps)-x_0(t)\right>\dot
  x_0(t),
\end{eqnarray*}
and Theorem~\ref{conv_thm1}.
\end{proof}

\section{A symmetric case}

In this section we consider the situation when the unperturbed
Hamiltonian system (\ref{np}) possesses the following symmetry
properties:
\begin{eqnarray}
 f_1(\xi_1,\xi_2)&=&f_1(-\xi_1,\xi_2),\label{s1}\\
 f_2(\xi_1,\xi_2)&=&-f_2(-\xi_1,\xi_2),\label{s2}\\
 (f_1)'_{(1)}(\xi_1,\xi_2)&=&-(f_2)'_{(2)}(\xi).\label{s3}
\end{eqnarray}
where $(h)'_{(i)}, i=1,2$ denotes the derivative of $h$ with respect
to the $i-$variable. The main consequence of this symmetry
assumption is given by the following lemma whose prove is immediate.

\begin{lemma}\label{sss} (\cite[Lemma 4.4]{tmms}) Assume $f\in
C^1(\mathbb{R}^2,\mathbb{R}^2) $ and that properties
(\ref{s1})-(\ref{s3}) hold true. Let $x_0$ be a nondegenerate cycle
of the Hamiltonian system (\ref{np}) and denote by $y$ the solution
of the linearized system (\ref{ls}) satisfying
\begin{equation}\label{Y}
  \left(\begin{array}{c} y_1(0)\\ y_2(0)\end{array}\right)=\left(\begin{array}{c} -\dot x_{0,2}(0)\\
\dot x_{0,1}(0)\end{array}\right).
\end{equation}
Then the functions
$$
  \widehat{z}(\theta)=\left(\begin{array}{c} y_2(\theta)
 \\ -y_1(\theta)\end{array}\right), \qquad \widetilde{z}(\theta)=\left(\begin{array}{c} -\dot x_{0,2}(\theta)\\
\dot x_{0,1}(\theta)\end{array}\right), \quad
 \theta\in\mathbb{R},
$$
where $\dot x_0(\theta)= (\dot x_{0,1}(\theta),\dot
x_{0,2}(\theta))$, are linearly independent solutions of the adjoint
system (\ref{ss}).
\end{lemma}

Lemma~\ref{sss} allows us to rewrite the multivalued map
$M^\bot:[0,T]\to K(\mathbb{R})$ defined in (\ref{M}) as follows
\begin{eqnarray*}
  M^\bot(t)&=&\{
  \dfrac{\dot x_{0,1}(t_*)}{y_1(T+t_*)}\int_{t-T}^t {\rm
  det}\left(-y(\tau),h(\tau)\right)d\tau: \\
  && h\in L^\infty([-T,T],\mathbb{R}^2):h(t)\in g(t,x_0(t),0)\ {\rm for\ a.a\
  }t\in[-T,T]\}.\label{M1}
\end{eqnarray*}
where $t_*\in [0,T]$ is such that $\dot x_{0,2}(t_*)=0.$ Therefore
Theorem~\ref{thm3} takes the form of the following
Theorem~\ref{thm4} when the symmetry assumptions
(\ref{s1})-(\ref{s3}) are satisfied. In particular, observe that the
statement of Theorem~\ref{thm4} refers only to the linearized system
(\ref{ls}) and not to the adjoint system (\ref{ss}).

\begin{theorem}\label{thm4} Assume $f\in
C^1(\mathbb{R}^2,\mathbb{R}^2)$ and
$g:\mathbb{R}\times\mathbb{R}^2\times[0,1]\to K(\mathbb{R}^2)$ upper
semicontinuous and satisfying (H).
 Let $x_\varepsilon$ be a $T$-periodic solution to
perturbed system (\ref{ps}) satisfying
\begin{equation*}
  \|x_\varepsilon(t)- {x}_0(t)\|\to 0\mbox{\quad as}\quad\varepsilon\to
  0
\end{equation*}
uniformly with respect to $t\in[0,T],$ where ${x}_0$ is a
nondegenerate $T$-periodic cycle of the Hamiltonian system
(\ref{np}). Let $y$ be the solution of the linearized system
(\ref{ls}) with the initial condition (\ref{Y}). Then there exists a
family $\{\Delta_\eps\}_{\eps>0}$ such that $\Delta_\eps\to 0$ as
$\eps\to 0$ and
$$
  x_\eps(t+\Delta_\eps)-x_0(t)\in\eps M^\bot(t)y(t)+  \left<\left(\begin{array}{c}  y_2(t)
 \\ -  y_1(t)\end{array}\right),x_\eps(t+\Delta_\eps)-x_0(t)\right>\dot
  x_0(t)+o(\eps),
$$
uniformly with respect to $t\in[0,T].$
\end{theorem}

\subsection*{Acknowledgement} The first author was supported by
the Grant BF6M10 of Russian Federation Ministry of Education and
U.S. CRDF (BRHE), by RFBR Grants 09-01-90407, 09-01-00468, by the
President of Russian Federation Young PhD Student grant
MK-1620.2008.1 and by "Researchers Mobility in the Field of
Scientific and Cultural Cooperation Programmes" of the University
of Modena and Reggio Emilia.

The second author was supported by the national research project
PRIN ``Ordinary Differential Equations and Applications''

Finally, the third author was supported by the national research
project PRIN ``Control, Optimization and Stability of Nonlinear
Systems: Geometric and Topological Methods''.


\begin{thebibliography}{99}
\bibitem{AC} Aubin,~ J.P. and Cellina,~ A.,  {\it Differential inclusions},
Grundlehren der mathematischen Wissenschaften, 264, Springer-Verlag,
1984.

\bibitem{awr} Awrejcewicz,~ J. and  Lamarque,~ C.-H.,  {\it Bifurcation and chaos in
nonsmooth mechanical systems}, World Scientific Series on Nonlinear
Science, Series A-45, World Scientific Publishing Co., 2003.

\bibitem{awr1} Awrejcewicz,~ J. and Holicke,~ M.M., {\it Smooth and
nonsmooth high dimensional chaos and the Melnikov-type methods},
World Scientific Series on Nonlinear Science, Series A-60, World
Scientific Publishing Co., 2007.

\bibitem{obu}
Borisovich,~Yu. G., Gelman,~B.D., Myshkis, A.D. and Obukhovski\u\i,
V.V., {\it Introduction to the theory of multivalued mappings},
Voronezhski\u\i $\;$ Gosudarstvenny\u\i $\;$ Universitet, Voronezh,
1986.

\bibitem{crit1}
Chicone,~ C. and Jacob,~ M.,  Bifurcations of critical periods for
plane vector fields, {\it Trans. Amer. Math. Soc.} 312 (1989),
433--486.

\bibitem{dem}
Demidovich,~ B.P., {\it Lectures on the mathematical theory of
stability} (in Russian). Izdat. Nauka, Moscow, 1967.

\bibitem{fec}
Fe\v ckan,~ M.,  Bifurcation of periodic solutions in differential
inclusions, {\it Appl. Math.} 42 (1997), 369--393.

\bibitem{fil}
Filippov,~ A. F.,  On some questions in the theory of optimal
regulation: existence of a solution of the problem of optimal
regulation in the class of bounded measurable functions (in
Russian). {\it Vestnik Moskov. Univ. Ser. Mat. Meh. Astr. Fiz. Him.}
(1959), 25--32.

\bibitem{fil1}
Filippov,~ A. F., {\it Differential equations with discontinuous
righthand side}, Mathematics and its Applications, Kluwer Academic
Publishers, Dordrecht, The Netherlands, 1988.

\bibitem{guck}
Guckenheimer,~ J. and Holmes,~ P., {\it Nonlinear oscillations,
dynamical systems, and bifurcations of vector fields}, Revised and
corrected reprint of the 1983 original. Applied Mathematical
Sciences No. 42, Springer-Verlag, New York, 1990.

\bibitem{nach}
Kamenskii,~ M. Makarenkov,~ O. and Nistri,~ P.,  A continuation
principle for a class of periodically perturbed autonomous systems,
{\it Math. Nachr.} 281 (2008), 42--61.

\bibitem{koz} Kamenskii,~ M.,  Obukhovskii,~V.V. and Zecca,~P.,
{\it Condensing Multivalued Maps and Semilinear Differential
Inclusions in Banach Space}, W.~deGruyter, Berlin, (2001).

\bibitem{kolm} Kolmogorov,~A.N. and Fom{\rm i}n,~S.V., {\it Elements of the
theory of functions and functional analysis} (in Russian). Fourth
edition, revised. Izdat. ``Nauka'', Moscow, 1976.

\bibitem{kol}
Kolovski\u\i,~M.Z., An application of the small-parameter method for
determining discontinuous periodic solutions,  in {\it Proc.
Internat. Sympos. Non-linear Vibrations} (in Russian.) Izdat. Akad.
Nauk Ukrain. SSR, Kiev. I (1961), 264--276.

\bibitem{kunze}
Markus,~K., {\it Non-smooth dynamical systems}, Lecture Notes in
Mathematics, No. 1744. Springer-Verlag, Berlin, 2000.



\bibitem{jmaa}  Makarenkov,~O. and Nistri,~P.,
Periodic solutions for planar autonomous systems with nonsmooth
periodic perturbations,  {\it J.~Math. Anal. Appl.} 338 (2008),
1401--1417.

\bibitem{cpaa}
Makarenkov,~O. and  Nistri,~P.,  On the rate of convergence of
periodic solutions in perturbed autonomous systems as the
perturbation vanishes, {\it Commun. Pure Appl. Anal.} 7 (2008),
49--61.

\bibitem{tmms}
Makarenkov,~O.,  Poincar\'e index and periodic solutions of
perturbed autonomous systems, {\it Trudy Moskov. Mat. Ob\v s\v c.}
70 (2008), in press.

\bibitem{mel}
Melnikov,~V. K.  On the stability of a center for time-periodic
perturbations (in Russian). {\it Trudy Moskov. Mat. Ob\v s\v c.} 12
(1963), 3--52.

\bibitem{per}
Perron,~O.,  Die Ordnungszahlen der Differentialgleichungssysteme,
{\it Math. Zeitschr.} 31 (1930), 748--766.

\bibitem{chic}
Rhouma,~M.B.H. and  Chicone,~C., On the continuation of periodic
orbits, {\it Methods Appl. Anal.} 7 (2000), 85--104.



\end{thebibliography}
\end{document}